\documentclass{article}

\usepackage{amsmath}
\RequirePackage{graphicx}
\RequirePackage{pifont,latexsym,ifthen,amsthm,rotating,calc,textcase,booktabs}
\RequirePackage{amsfonts,amssymb,amsbsy,amsmath}

\theoremstyle{definition}
\newtheorem{theorem}{Theorem}[section]

\newtheorem{definition}[theorem]{Definition}

\newcommand\etal{\emph{et~al.}}

\newcommand\acks{\section*{Acknowledgements}}

\makeatletter
\@twosidetrue
\makeatother

\begin{document}

\author{Jonathan Fine\\
15 Hanmer Road, Milton Keynes, MK6 3AY, United Kingdom}
\title{Vassiliev-Kontsevich invariants and Parseval's theorem}

\maketitle

\newtheorem{problem}[theorem]{Problem}

\begin{abstract}
We use an example to provide evidence for the statement: the
Vassiliev-Kontsevich invariants $k_n$ of a knot (or braid) $k$ can be
redefined so that $k = \sum_0^\infty k_n$. This constructs a knot from
its Vassiliev-Kontsevich invariants, like a power series expansion.
The example is pure braids on two strands $P_2\cong \mathbb{Z}$, which
leads to solving $e^\tau=q$ for $\tau$ a Laurent series in $q$. We set
$\tau = \sum_1^\infty (-1)^{n+1} (q^n - q^{-n})/n$ and use Parseval's
theorem for Fourier series to prove $e^\tau=q$.  Finally we describe
some problems, particularly a Plancherel theorem for braid groups,
whose solution would take us towards a proof of $k=\sum_0^\infty k_n$.
\end{abstract}

\section{Introduction}

Throughout we think of knots as being in $\mathbb{R}^3$ and braids as
being elements of a group. Sometimes we will say knot when we mean an
isotopy class of knots, and a braid when we mean a realisation of a
braid in $\mathbb{R}^3$.  Often we will need finite, and sometimes
convergent infinite, formal sum of knots or braids.  The context will
usually make clear which is meant. For example, in $k = \sum_0^\infty
k_n$ the quantity $k$ is the isotopy class of a knot, and each $k_n$
is a convergent formal sums of isotopy classes of knots.  Usually, $b$
will refer to a realisation of a braid.

The Vassiliev-Kontsevich invariant \cite{kont:vki,bn:vki} $b_n=b_n(b)$
of a braid $b$ can be calculated by using the height $h$ to slice the
$b$ into slices and then performing an interated $n$-slice integral
over the simplex $1\geq h_1> \ldots \geq h_n \geq 0$.  The integrand
measures the `twistyness' of the slice, and composition of braids is
used to glue the slices together.  Each invariant $b_n$ lies in a
finite-dimensional vector space, which is usually taken to be a
quotient $V_n/V_{n+1}$ in the Vassiliev filtration (see
\cite{vav:cks,vav:ctd} and Section~3 below).

To calculate $k_n$ of a knot $k$ the same method can be used, except
that the height function $h$ will have critical points, each of which
makes a contribution that is glued into answer.  In this paper we will
use $*$ to denote, as appropriate, either the group law for braids or
the connected sum operator for knots. We can also define $*$ on the
$k_n$. In particular, if $k$ and $k'$ are two knots (or braids) it
then follows that $(k*k')_n = \sum_{i+j=n}k_i * k'_j$.

It is not known if a knot $k$ is determined by its invariants $k_n$.
We approach this problem by finding a space $\mathcal{K}$ which
contains $k$, and then lifting $k_n$ from $V_n/V_{n+1}$ and into
$\mathcal{K}$.  One can then ask if $k =\sum k_n$.  We show that this
approach works for braids on two strands and suggest how it might be
extended to more strands and to knots.

Throughout let $q$ be a generator for the group $P_2\cong\mathbb{Z}$
of pure braids on two strands.  We think of $q$ as two strands
rotating around each other in $\mathbb{R}^3$.  Because each slice is
simply a rotation of any other, the integrand is constant. It follows
that the integrand for $b_n = b_n(q)$ is the $n$-fold $*$-product
$t^n$ of the integrand $t$ for $b_1(q)$. The region of integration is
the unit $n$-simplex, with volume $1/n!$, and so $b_n(q) = t^n/n!$ and
thus at least formally $\sum_0^\infty b_n = e^t$.

For $k=\sum k_n$ to hold, the integrand must be special.  In
particular, it must be a sum of knots (or braids).  Some simple
calculations, which we omit, show that the sum must be infinite and so
questions of convergence arises.  Throughout we will use $\mathcal{K}$
to denote formal infinite sums of (isotopy classes of) knots, whose
coefficients are $L^2$-convergent and similarly $\mathcal{P}_m$ for
$P_m$.  We use $\mathcal{S}\cong S * \mathcal{P}_m$ to denote `pure
braid changes' to a slice $S$ on $m$ strands.

Recall that we wish to solve $q = \sum b_n(q)$, as a special case of
$k = \sum k_n$.  Let us now write $\tau$ for $b_1(q)$.  The problem
now amount to solving $q = e^\tau := \sum \tau^n/n!$ for $\tau$ in a
vector space that also contains $\tau^n$, for $n>1$.  To obtain a
candidate for the solution $\tau\in\mathcal{P}_2$ we use a trick.
Write $p=q^{-1}$.  We can write $q = (1+q)/(1+p)$ and so at least
formally our candidate is $\ln(1+q) - \ln(1+p)$.

\begin{definition}
\label{def:tau}
\begin{equation*}
\tau = \sum_1^\infty (-1)^{n+1}(q^n-p^n)/n \in \mathcal{P}_2
\end{equation*}
Because $\sum_1^\infty 1/n^2$ is absolutely convergent, $\tau$ is in
$\mathcal{P}_2$.  Note that $f(z) = \sum_1^\infty
(-1)^{n+1}(z^n-z^{-n})/n$ is nowhere absolutely convergent.
\end{definition}

\section{Proof of $e^\tau = q$}

Earlier we saw that this is a special case of $k = \sum k_n$.  In this
section we write $\mathcal{P}_2$ as $L^2(\mathbb{Z})$.  We will prove

\begin{theorem}
  \label{thm:q=exp(tau)}
  $\tau \in L^2(\mathbb{Z})$, as defined in Definition~\ref{def:tau},
  satisfies $\exp(\tau) = q$.
\end{theorem}

This is a shorthand for saying first that the convolutions $\tau,
\tau^2, \tau^3, \ldots$ all lie in $L^2(\mathbb{Z})$ and second that
the sum $1 + \tau + \tau^2/2! + \ldots$ converges to $q\in
L^2(\mathbb{Z})$.  To prove this result we use Fourier series and
Parseval's theorem.

For any integrable function $f$ defined on $[-\pi, \pi]$ we as usual
let
\begin{equation*}
  c_n(f) = \frac{1}{2\pi}\int_{-\pi}^{\pi} e^{-in\theta}f(\theta)\,d\theta
\end{equation*}
denote the $n$-th complex Fourier coefficient of $f$.  We now state

\begin{theorem}[Parseval's theorem]
Let $A(x)$ and $B(x)$ be integrable functions on $[-\pi, \pi]$ with
complex Fourier coefficients $a_n$ and $b_n$.  Then
\begin{equation*}
\sum_{-\infty}^\infty
a_n\overline{b_n} = \frac{1}{2\pi}\int_{-\pi}^\pi
A(x)\overline{B(x)}\,dx\>.
\end{equation*}
\end{theorem}

For the function $f(\theta)=\theta$ we have
\begin{align*}
  c_n(f) &= \frac{1}{2\pi} \int_{-\pi}^\pi e^{-in\theta}\theta\, d\theta
  = \left.\frac{i}{2n\pi}e^{-in\theta}\theta\right|_{-\pi}^{\pi}
  - \frac{i}{2n\pi}\int_{-\pi}^\pi e^{-in\theta}\, d\theta
  \\
  &=\frac{i(-1)^{n}}{n}
\end{align*}
for $n\ne 0$, while $c_0(f) = \int_{-\pi}^\pi\theta\,d\theta = 0$.
Thus, as a series $\tau$ is the Fourier transform of $i\theta$.

We can extend this result as follows (the proof will come later).  For
$\psi$ in $L^2(\mathbb{Z})$ we use $c_n(\psi)$ to denote $\psi_n$,
which we also interpret as the coefficient of $q^n$.

\begin{theorem}
\label{thm:c_n(tau^m)}
\begin{equation*}
  c_n(\tau^m) = \frac{1}{2\pi}\int_{-\pi}^{\pi}\> e^{-in\theta}\>(i\theta)^m \> d\theta
\end{equation*}
\end{theorem}

\begin{proof}[{{Proof of Theorem \ref{thm:q=exp(tau)}}}]
The algebraic part of the proof, which relies on
Theorem~\ref{thm:c_n(tau^m)}, is
\begin{align*}
  c_n(\exp(\tau)) &= \sum \frac{c_n(\tau_m)}{m!}\\
  &= \frac{1}{2\pi}\sum\int_{-\pi}^{\pi}e^{-in\theta}\frac{(i\theta)^m}{m!}\, d\theta\\
  &= \frac{1}{2\pi}\int_{-\pi}^{\pi}e^{-in\theta}\sum \frac{(i\theta)^m}{m!}\, d\theta\\
  &= \frac{1}{2\pi}\int_{-\pi}^{\pi}e^{-in\theta} e^{i\theta}\, d\theta\\
  &= \frac{1}{2\pi}\int_{-\pi}^{\pi}e^{i(1-n)\theta}\, d\theta
\end{align*}
and hence $c_1=1$ and $c_n=0$ otherwise.  The analytic part is that
the sum-integral is absolutely convergent and so, by Fubini's theorem,
we can perform the integration first (which then allows us to simplify
the sum).
\end{proof}

\begin{proof}[{{Proof of Theorem \ref{thm:c_n(tau^m)}}}]
We rewrite the result to be proved as
\begin{equation*}
  c_n(\tau^m) = \frac{1}{2\pi}\int_{-\pi}^{\pi}\> 
  (i\theta)^{m-1}
  \times
  (i\theta)
  \>
  e^{-in\theta}
  \>
  d\theta
\end{equation*}
and apply Parseval's theorem with $A=(i\theta)^{m-1}$ and $B=
\overline{i\theta e^{-in\theta}}$ (and an induction hypothesis). This
tells us that the right hand side is equal to
$
  \sum c_k(\tau^{m-1})\overline{c_k(B)}
$
and as
\begin{equation*}
  c_k(\overline{i\theta e^{-in\theta}}) 
  = \frac{1}{2\pi}\int_{-\pi}^{\pi}\overline{i\theta} e^{in\theta}e^{-ik\theta}d\theta
  = c_{n-k}(\tau)
\end{equation*}
the result follows.
\end{proof}

\section{Taking values in $\mathcal{P}_m$, not $V_n/V_{n+1}$}

Here we discuss how to extend the main result to $P_3$.  This will
also help us understand better the result for~$P_2$.  Prior knowledge
of Bar-Natan's paper~\cite{bn:vki} would help the reader.  Here we
outline the standard construction, but draw attention to differences.
Recall that the Vassiliev-Kontsevich invariants can be evaluated by
gluing together slice contributions. In $P_2$ each slice is
effectively the same as any other, and $P_2$ is commutative.  This
makes the definition of $\tau$ quite simple.

For $P_3$ the r\^ole of the slice is not so clear.  We presented
$e^\tau=q$ as a calculation of $q$ from its Vassiliev-Kontsevich
invariants.  (This is the \emph{inverse problem} to computing $b_n$
from $b$.)  To compute $b_1(b)$ of a braid $b$ one divides $[0, 1]$
into slices and sum the contributions made by each slice.  This
contribution uses $\int dt/(z_1-z_2)$ to measure the twist in the
slice.  But we want, for example, $b_1(q)$ to be $\tau$.  This can be
obtained by adding a factor of $\tau$ to the integrand.  However, this
factor must be introduced geometrically, as a slice contribution
(see Figure~\ref{fig:tau}).

\begin{figure}[htp]
\includegraphics[height=0.4\vsize]{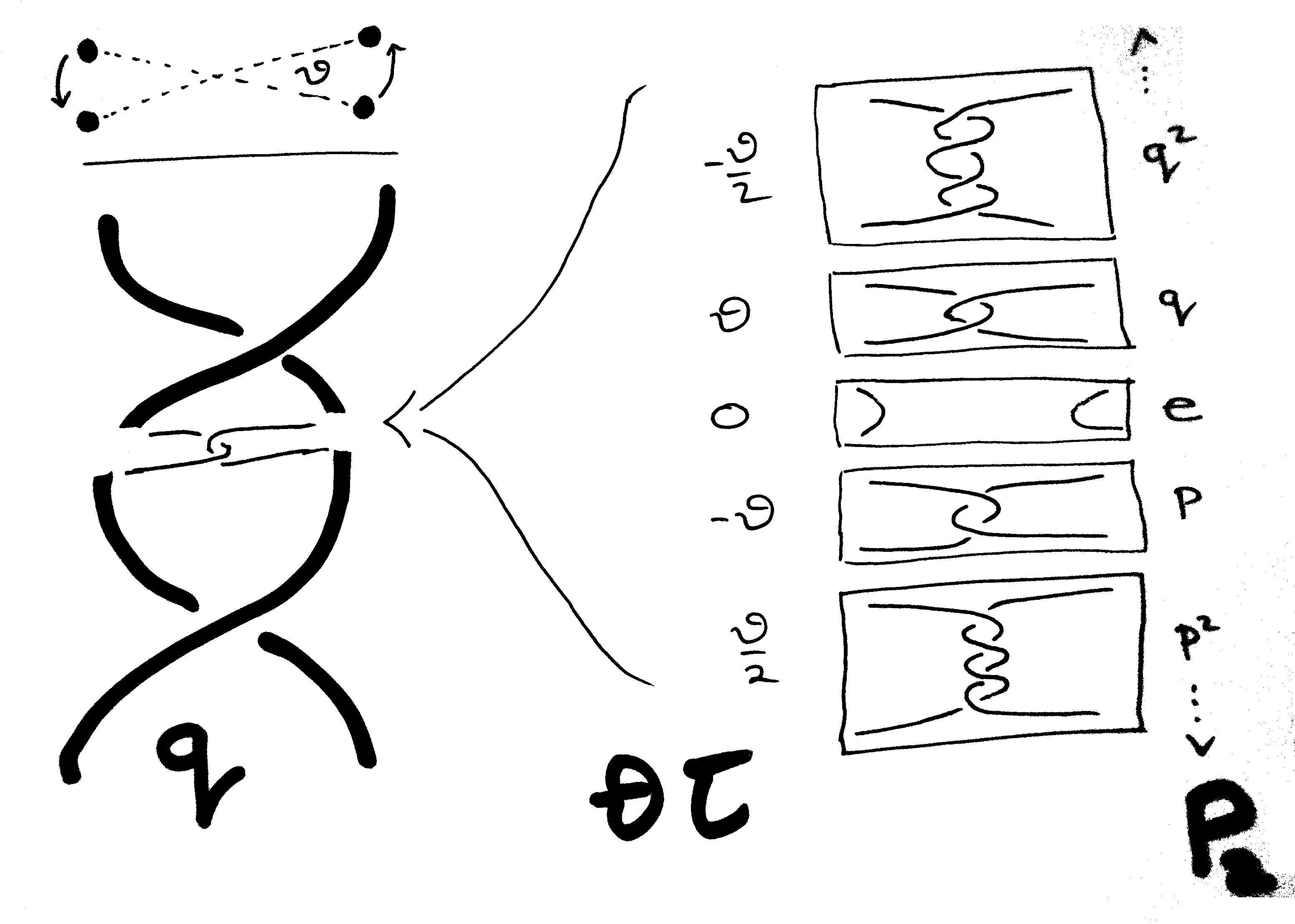}
\caption{$e^\tau = q$.}\label{fig:tau}
\end{figure}

In the usual Kontesevich definition each $b_n$ lies in a finite
dimensional vector space, which can be taken to be the quotient
$V_n/V_{n+1}$ in the Vassiliev `braids with $n$ double points'
filtration of the vector space of finite formal sums of braids (also
known as the group ring).  

\newpage

\begin{figure}[htp]
\includegraphics[height=0.3\vsize]{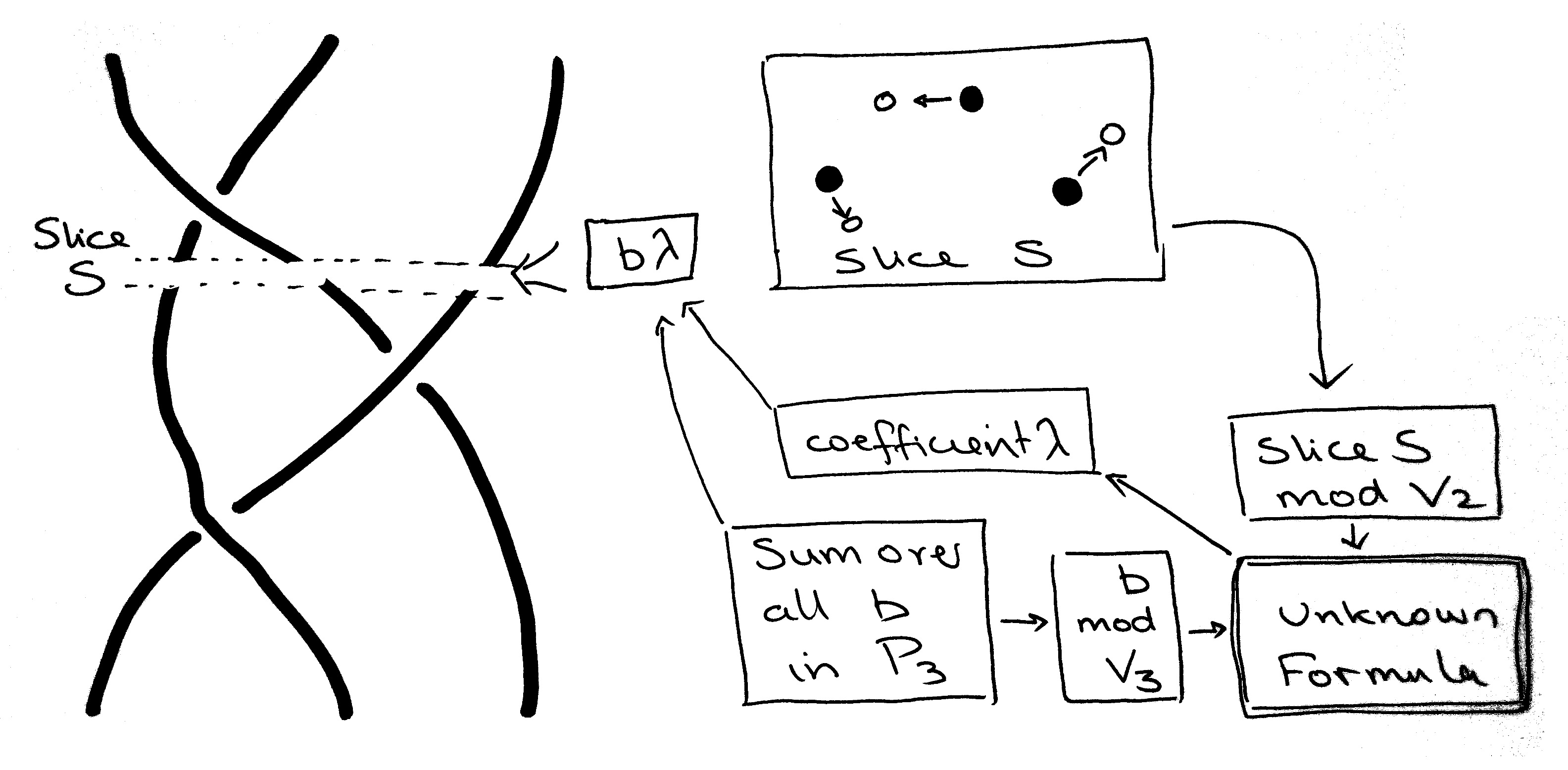}
\caption[]{$b = \sum b_n$?  Each slice contribution contains all
  pure braid modifications, with coefficient $\lambda$. }\label{fig:mult}
\end{figure}

Now suppose we have a slice $S$ with $m\geq 3$ strands in it.  Each
relative motion between a pair of strands contributes to the slice.
In the usual definition the integrand value space for this
contribution is $V_1/V_2$.  To achieve $k=\sum k_n$ we require
$\mathcal{S}\cong S*\mathcal{P}_m$ as the value space~(see
Figure~\ref{fig:mult}).  This is an important difference.

When the value space is $V_1/V_2$ we can ignore the other strands when
we compute the contribution made by a pair.  But some simple examples
(not given here) show that when $\mathcal{S}$ is the value space we
have to link in the other strands, and it seems likely that
\emph{every} element of $S*P_m$ will so appear.

Note that in $\mathcal{P}_2$ the difference $q^0 - 1/n\sum_1^nq^i$
lies in the Vassiliev subspace $V_1$ and in $\mathcal{P}_2$ the
corresponding sequence converges to $q^0$.  The same argument also
shows that in $\mathcal{K}$ and $\mathcal{P}_m$ the Vassiliev
subspaces are dense.  The Kontsevich invariants are an analogue of
differentiation, which is well known not to be a continuous operator
on $L^2$ spaces.

\section{Problems}

Here we state some problems related to proving $k =\sum k_n$.

\begin{problem}
Suppose we have a slice $S$ with $m$ strands.  What is the
contribution, which lies in $\mathcal{S}\cong S*\mathcal{P}_m$, of that
slice?  In particular, for each $b \in P_m$ what is coefficient of $b$
in the slice contribution?
\end{problem}

Note that $b$ is a member of a braid group, while $S$ is (part of) the
realisation of a braid.  Here the difference is important.  We have
already solved this problem, in the case of two strands.  Let $\theta$
be the twisting or `fractional winding number' of the two strands and
let $b$ be an element of $P_2$.  We know that if $b=q^n$ then the
contribution is $\theta \times (-1)^{n+1}/n \times b$.  To extend the
main result to $P_3$ we need a similar formula for each $b\in P_3$.

\begin{problem}
Suppose slice $S$ has three strands and $b$ is in $P_3$.  Produce a
formula that depends on the pairwise twisting in $S$ and also say
$b_1(b)\in V_1/V_2$ and $b_2(b)\in V_2/V_3$ that generalises the
two-strand case.  (See Figure~\ref{fig:mult}.)
\end{problem}

Here is a hint.  One might expect the twistyness $\theta$ in the slice
to be divided into parts, with each part going to some $b$ in $P_3$.
In particular, might be looking to solve $\sum c(b) = \theta$, where
the sum is over the elements $b$ of $P_3$.  However, in the case of
$P_2$ the twistyness of $q_n$ is (of course) $n$, and so the
corresponding sum is $\sum_{n=1}^\infty (-1)^{n+1}(n - (-n))/n =
(1-1)-(1-1)+ \ldots$ which imay be best thought of as $-2\zeta(0)$,
where $\zeta$ is the analytic continuation of $\sum_1^\infty n^{-s}$,
and $\zeta(0)=-1/2$.  (Similarly, naively applying $b_m(q^n) = n^m/m!
(b_1(q)^m$ to $\tau$ leads to the divergent sum $\sum n^{m-1}$.)

There is in addition a constraint.  A realisation in $\mathbb{R}^3$ of
a braid $b$ on $n$ strands can be deformed into another realisation.
This should not change the value of say $b_2(b)\in\mathcal{S}$.  When
$b_2$ takes values in $V_2/V_3$ this is a consequence of the integrand
satisfying the Arnold identity~\cite[\S4.2]{bn:vki}.  When we use
$\mathcal{S}$ this makes this constraint considerably more exacting.
It seems to require every element of $P_m$ to appear.  We can add
critical points to the representation of a knot by adding an
\emph{N}-shaped kink in to a vertical line.  This does not, of course,
change knot invariants.  Thus, in addition to the slices, the critical
points also contribute.  Bar-Natan \etal{}~\cite{bn-etal:wheels} have
found an explicit formula for this contribution, when values are taken
in $V_n/V_{n+1}$.  They call this `wheeling', from the shape of some
diagrams used.

\begin{problem}
Extend wheeling so that it work for $\mathcal{K}$.\label{problem:wheeling}
\end{problem}

Here are two braid group questions.

\begin{problem}
Let $a$ and $b$ be any two elements in $\mathcal{P}_3$, the space of
$L^2$ formal sums of braids in $P_3$.  Is the product $a * b$
absolutely convergent?
\end{problem}

\begin{problem}
Is there a Plancherel theorem for $P_3$?
\end{problem}

Here are two more general questions.

\begin{problem}
Drinfeld's associator~\cite{drin:qtqh} is an alternative to the Kontsevich's
integral approach. Is there a way of refining it to produce values in
$\mathcal{P}_n$?
\end{problem}

\begin{problem}
Is $k = \sum k_n$ a new connection between the mathematics of knots
and quantum field theory?
\end{problem}

\section{Summary}

We saw that the problem $k=\sum k_n$ for knots leads to solving
$e^\tau=q$, whose solution relies on Parseval's theorem.  There is a
local description of this, in terms of the contribution made by the
slices in the Kontsevich integral. If each slice $S$ made a suitable
contribution lying in the braid modifications $\mathcal{S}\cong
S*\mathcal{P}_m$ of $S$ we would have $b=\sum b_n$ for
braids. Further, if the wheeling at critical points can be similarly
extended, then we have $k = \sum k_n$ for knots.  In the previous
section we described some problems that would need to be solved, for
this program to be carried out.

\acks 

I thank Phil Rippon for help with the proof of
Theorem~\ref{thm:q=exp(tau)}, and Joel Fine for reading an earlier
version of this paper. Any remaining errors are mine.

\bibliographystyle{amsplain}
\bibliography{braids}

\end{document}